\documentclass[ECP]{ejpecp} 

\usepackage{pdfsync}
\usepackage{a4wide}
\usepackage{graphics}
\usepackage{amssymb,latexsym}
\usepackage{array}
\usepackage[scaled=0.95]{helvet}
\usepackage{multicol}
\usepackage{textcomp}
\usepackage{indentfirst}
\usepackage{array}
\usepackage{nccmath}
\usepackage[T1]{fontenc}
\usepackage[utf8]{inputenc}
\usepackage{verbatim}
\usepackage{courier}
\usepackage[english]{babel}
\usepackage{paralist}
\usepackage{tikz}
\usepackage{epsfig,psfrag}

\SHORTTITLE{The common ancestor type distribution} 

\TITLE{The common ancestor type distribution \\ of a $\Lambda$-Wright-Fisher process with selection and mutation}

\AUTHORS{%
	Ellen~Baake\footnote{Faculty of Technology, Bielefeld University, Box 100131, 33501 Bielefeld, Germany.
		\EMAIL{ebaake@techfak.uni-bielefeld.de}}
	\and 
	Ute~Lenz\footnote{Institut f\"{u}r Mathematik, Goethe-Universit\"{a}t Frankfurt, Box 111932, 60054, Frankfurt am Main, Germany. 
		\EMAIL{lenz@math.uni-frankfurt.de, wakolbinger@math.uni-frankfurt.de}}
	\and 
	Anton~Wakolbinger${}^\dagger$
		}


\AMSSUBJ{60J75; 92D15; 60C05; 05C80} 

\ARXIVID{1603.03605} 

\VOLUME{0}
\YEAR{2012}
\PAPERNUM{0}
\DOI{?}

\ABSTRACT
{Using graphical methods based on a
`lookdown'
and pruned version of the {\em ancestral selection graph}, we obtain a 
representation of the type distribution of the ancestor in a two-type Wright-Fisher 
population with mutation and selection, conditional on the overall type frequency in the old population. 
This extends results from \cite{Unserpaper} to  the case of heavy-tailed offspring, directed by a reproduction measure 
$\Lambda$. 
The representation is in terms of the equilibrium tail probabilities of the line-counting process $L$ of the graph. We identify a strong pathwise Siegmund dual of $L$, and characterise the equilibrium tail probabilities of $L$ in terms of hitting probabilities of the dual process. 
}

\begin{document}



\section{Introduction}

We consider a Wright-Fisher process with two-way mutation and selection. This is a classical model of mathematical population genetics, which describes the evolution, forward in time, of the type composition of a population with two types. Individuals reproduce and change type, and the reproduction rate depends on the type (the beneficial type reproduces faster than the less favourable one).

In a previous paper \cite{Unserpaper}, we have presented a graphical construction, termed the \emph{pruned lookdown ancestral selection graph (p-LD-ASG)}, which allows us to 
identify the common ancestor of a population in the distant past, and to represent its type distribution. This construction keeps track of the collection of all \emph{potential} ancestral lines of an individual.
As the name suggests, the p-LD-ASG combines elements of  the \emph{ancestral selection graph (ASG)} of Krone and Neuhauser \cite{KroneNeuhauser} and the \emph{lookdown construction} of Donnelly and Kurtz \cite{DonnellyKurtz}, which here leads to a hierarchy that encodes who is the \emph{true} ancestor once the types have been assigned to the lines.
In addition, a pruning procedure is applied to reduce the graph.

A key quantity is the process $L$, which counts the number of potential ancestors at any given time. The \emph{ancestral} type distribution is expressed in terms of the stationary distribution of $L$ together with the \emph{overall} type distribution in the past population. The two distributions may be substantially different. This mirrors the fact that the true ancestor is an individual that is successful in the long run; thus, its type distribution is biased towards the favourable type. Explicitly, the ancestral type distribution is represented as a series  in terms of the frequency of the beneficial type in the past, where the coefficients are the tail probabilities of the stationary distribution of $L$ and are known in terms of a recursion.

The results obtained so far referred to Wright-Fisher processes. These arise as scaling limits of processes in which an individual that reproduces has a single offspring that replaces a randomly chosen individual (thus keeping population size constant); in the ancestral process, this corresponds to a coalescence event of a pair of individuals. Here we will consider a natural generalisation, the so-called $\Lambda$-Wright-Fisher processes. These include reproduction events where a fraction $z>0$ of the population is replaced by the offspring of a single individual; this leads to \emph{multiple merger events} in the ancestral process.

The $\Lambda$-Wright-Fisher processes belong to the larger class of $\Lambda$-Fleming-Viot processes (which also include multi-(and infinite-)type generalisations). These, together with their ancestral processes, the so-called $\Lambda$-coalescents, have become  objects of intensive research in the past two decades. Although less is known for the case with selection, progress has been made in this direction as well  (see for example \cite{BahPardoux,Depperschmidtetal,DonnellyKurtz,EtheridgeGriffithsTaylor,Griffiths}).

Besides deriving our main result on the {\em common ancestor type distribution} of a $\Lambda$-Wright-Fisher process (stated in Sec.~\ref{sec:model_result}), the purpose of our paper is twofold: First, we will extend the p-LD-ASG to include multiple-merger events; this will lead to the \emph{p-LD-}$\Lambda$\emph{-ASG}.
Second, in the footsteps of Clifford and Sudbury \cite{CliffordSudbury}, we will construct a \emph{Siegmund dual} of the line-counting process $L$ of the p-LD-$\Lambda$-ASG. 
In line with a classical relation between entrance laws of a monotone process and exit laws of its Siegmund dual (discovered by Cox and R\"{o}sler \cite {CoxRoesler}), the tail probabilities of $L$ at equilibrium correspond to   hitting probabilities of the Siegmund dual.  
This Siegmund dual is a new element of the analysis: In \cite{Unserpaper}, the recursions for the tail probabilities were obtained from the generator of $L$, in a somewhat technical manner. The duality provides a more conceptual approach, which is interesting in its own right, and yields the recursion in an elegant way, even in the more involved case including multiple mergers. It will also turn out that the Siegmund dual of $L$ is a natural generalisation (to the case with selection) of the so-called {\em fixation line} (or fixation curve), introduced by Pfaffelhuber and Wakolbinger  \cite{PfaffelhuberWakolbinger} for Kingman coalescents and investigated by H\'{e}nard \cite{olivier} for $\Lambda$-coalescents.

The paper is organised as follows. In Section~\ref{sec:model_result}, we recapitulate the $\Lambda$-Wright-Fisher model with mutation and selection, and the corresponding ancestral process, the $\Lambda$-ASG; we also provide a preview of our main results. In Section~\ref{sec:pruned}, we extend the p-LD-ASG to the case with multiple mergers. Section~\ref{sectionsiegmunddual} is devoted to the Siegmund dual. 
The dynamics of this dual process is identified via a pathwise construction and thus yields a strong duality. Once the dual is identified, it leads to the tail probabilities of $L$ with little effort.

\section{Model and main result} \label{sec:model_result}

We will consider a population consisting of individuals each of which is either of deleterious type (denoted by 1) or of beneficial type (denoted by 0). The population evolves according to random reproduction, two-way mutation, and fertility selection (that is,  the beneficial type reproduces at a higher rate), with constant population size over the generations. The parameters of the model are 
\begin{itemize}
	\item the {\em reproduction measure} $\Lambda$, which is a probability measure on $[0,1]$, and whose meaning will be explained along with that of the generator $G_{X} $ below Eq.~\eqref{generatorlambdawf},
	\item {\em the selective advantage $\sigma$} (a non-negative constant that quantifies the reproductive advantage of the beneficial type and is scaled with population size),
	\item  the {\em mutation rates} $\theta \nu_0$ and $\theta \nu_1$, where $\theta, \nu_0$, and $\nu_1$ are non-negative constants with $\nu_0+\nu_1=1$. Thus, $\nu_i$, $i\in \{0,1\}$, is the  probability that the type is $i$ after a mutation event; note that this includes \emph{silent events}, where the type remains unchanged. 
\end{itemize}
We will work in a scaling limit in which  the population size is infinite and  time is scaled such that the rate at which a fixed pair of individuals takes part in a reproduction event is 1.  The process 
$X:=(X_t)_{t \in \mathbb{R}}$ describing the type-$0$ frequency in the population then has the generator (cf.\ \cite{EtheridgeGriffithsTaylor,Griffiths})

\begin{equation}
\begin{split}
G_{X} g(x) &= \int_{(0,1]}^{} \big[x(g(x+z(1-x))-g(x)) + (1-x)(g(x-zx)-g(x))\big]\frac{\Lambda(dz)}{z^2}\\
& \qquad + \Lambda(\{0\})\cdot \frac{1}{2}x(1-x)g''(x) \ +\ \big[\sigma x(1-x) - \theta\nu_1 x + \theta\nu_0 (1-x) \big]g'(x).  
\end{split}
\label{generatorlambdawf}
\end{equation}

The first and second terms of this generator describe the neutral part of the reproduction.
In the case $\Lambda=\delta_{0}$ (to which we refer as the \textit{Kingman case}), the first term  vanishes and $X$ is a Wright-Fisher diffusion with selection and mutation.
Concerning the part of $\Lambda$  concentrated on $(0,1]$, 
the  measure $dt\, \Lambda(dz) / z^2$ figures as intensity measure of a Poisson process, where a point $(t ,z)$, $t\in \mathbb{R},z\in (0,1]$, means that at time $t$ a fraction $z$ of the total population is replaced by the offspring of a randomly chosen individual. Consequently, if the fraction of type-$0$ individuals is $x$ at time $t-$, then at time $t$  the frequency of type-$0$ individuals in the population is $x+z(1-x)$ with probability $x$ and $x(1-z)$ with probability $1-x$. The third term of generator \eqref{generatorlambdawf} describes the systematic (logistic) increase of the frequency $x$ due to selection, and the type flow due to mutation.

In the absence of both selection and mutation (i.e. when $\sigma = \theta = 0$), the moment dual of the $\Lambda$-Wright-Fisher process is the line-(or block-)counting process of the \textit{$\Lambda$-coalescent}. The latter was introduced independently  by Pitman \cite{pitman1999}, Sagitov \cite{Sagitov}, and Donnelly and Kurtz \cite{DonnellyKurtza}, see  \cite{berestycki2009} for an introductory review.

The rate at which any given tuple of $j$ out of $b$ blocks merges into one is
\begin{equation}
\lambda_{b,j}:= \int_{0}^{1} z^j(1-z)^{b-j} z^{-2} \Lambda(dz), \quad j \leq b.
\label{mergerrateofkblocks}
\end{equation}
Thus the transition rate of the  line-counting process 
from state $b$ to state $c< b$ is given by $\binom{b}{b-c+1} \lambda_{b,b-c+1}$.
Note that   $\Lambda =\delta_0$ corresponds to Kingman's coalescent; here, $\lambda_{b,j} =\delta_{2,j}$ for all $b\ge 2$. The measure $\Lambda$ is said to have the property CDI if the  $\Lambda$-coalescent {\em comes down from infinity}, i.e. $\infty$ is an entrance boundary for its line-counting process. 

 When \textit{selection} is present (i.e. $\sigma >0$), an additional component of the dynamics of the genealogy must be taken into account. In this case, in addition to the  (multiple) coalescences just described, the lines (or blocks) may also undergo a binary branching at rate $\sigma$ per line. The resulting branching-coalescing system of lines is a straightforward generalisation of the ancestral selection graph (ASG) of Krone and Neuhauser \cite{KroneNeuhauser} to the multiple-merger case; we will  call it the \textit{$\Lambda$-ASG}. The $\Lambda$-ASG belonging to a sample of $n$ individuals taken from the population at time $\overline{t}=0$ describes all potential ancestors of this sample at times $t < 0$.
Throughout we use the variables $t$ and $r$ for  \emph{forward} and \emph{backward} time, respectively.

We denote the line-counting process of the $\Lambda$-ASG by $K=(K_r)_{r \geq 0 }$. It takes values in $\mathbb{N}$ and its generator is
\begin{equation}
G_K g(b) =  \sum_{c=1}^{b-1}{\binom{b}{b-c+1}} \lambda_{{b},{b-c+1}} \left[g(c)-g(b)\right] + b\sigma \left[g(b+1)-g(b)\right].
\label{generatorlambdaasg}
\end{equation}
The process $K$ is the moment dual of the $\Lambda$-Wright-Fisher process with selection coefficient~$\sigma$ and mutation rate $\theta=0$, in the sense that
\begin{equation}\label{momentdual}
\mathbb E[(1-X_t)^n \, | X_0=x] = \mathbb E[(1-x)^{K_t}\, | \, K_0 = n],
\end{equation}
see e.g.  \cite[Thm. 4.1]{EtheridgeGriffithsTaylor}. 

Throughout we will work under the
\begin{assumption}\label{assumptioncdi}
$0\le \sigma < \sigma^\ast:=  -\int_0^1 \log(1-x)\frac{\Lambda(dx)}{x^2}.$
\end{assumption}
Combining results of  \cite{Foucart} and \cite{Griffiths}, one infers that Assumption \ref{assumptioncdi} is equivalent to the positive recurrence of the process $K$ on $\mathbb N$. Indeed, it is proved in \cite[Theorem 3]{Griffiths}  (for the case $\sigma^\ast < \infty$) and \cite[Theorem 1.1]{Foucart}  (for the case $\sigma^\ast = \infty$)    that Assumption \ref{assumptioncdi} is {\em equivalent}  to $\mathbb P[X_\infty =1  \mid X_0=x] < 1$ for all $x<1$, where $X_\infty$ denotes the a.s. limit of $X_t$ as $t\to\infty$. Combined with the moment duality \eqref{momentdual}, this readily implies that Assumption \ref{assumptioncdi} is equivalent to the positive recurrence of $K$ on $\mathbb N$ if $\sigma > 0$.  

A direct proof that Assumption \ref{assumptioncdi} {\em implies} the positive recurrence of $K$ on $\mathbb N$ in the case $\sigma > 0$  is provided by \cite[Lemma 2.4]{Foucart}. (Note in this context that $K$ is clearly non-explosive because it is dominated by a pure birth process with birth rate $b\sigma$, $b \in \mathbb N$; this makes the first assumption in \cite[Lemma 2.4]{Foucart} superfluous).

For $\sigma = 0$, the process $K$, when started in $b \in \mathbb N$, is eventually absorbed in $1$. This complements the previous argument in showing that under Assumption \ref {assumptioncdi} the process $K$ has a unique equilibrium distribution and a corresponding time-stationary version indexed by $r\in \mathbb R$. Similarly, there exists a time-stationary version of the $\Lambda$-ASG, which we call the {\em equilibrium $\Lambda$-ASG}, and which will be a principal object in our analysis.
\begin{remark}\label{recurrence}
It is proved in  \cite{HerrigerMoehle} that $\sigma^\ast = \lim_{k\to \infty} \frac{\log k}{\mathbf E_{k}[T_{1}]}$, where $T_1$ is the first time at which the line-counting process of the $\Lambda$-coalescent hits $1$. In particular, 
if the measure $\Lambda$ has the property CDI, then $\sigma^\ast = \infty$  and hence Assumption \ref{assumptioncdi} is satisfied for all $\sigma \ge 0$.
\end{remark}

\textit{Mutations} can be superimposed as independent point processes on the lines of the $\Lambda$-ASG: On each line, independent Poisson point processes of mutations to type $0$ (`beneficial mutation events') come at rate $\theta\nu_0$ and to type $1$ (`deleterious mutation events') at rate $\theta\nu_1$.

For $\underline{t} < \overline{t}$ and for a given frequency $x$ of type-$0$ individuals in the population at time $\underline{t}$, the $\Lambda$-ASG may be used to determine the types in a sample $\mathcal S$ taken at time $\overline{t}$, together with its ancestry between times $\underline{t}$ and $\overline{t}$, by the following generalisation of the procedure in \cite{KroneNeuhauser}. 
Each line of the $\Lambda$-ASG at time $\underline{t}$ is assigned type $0$ with probability $x$ and type $1$ with probability $1-x$, in an iid fashion. Let the types then evolve forward in time along the lines: after each beneficial or deleterious mutation, the line takes type $0$ or $1$, respectively. At each neutral reproduction event (which is a coalescence event backward in time), the descendant lines inherit the type of the parent. This is also true for the (potential) selective reproduction events (the branching events backward in time), but here one first has to decide which of the two lines is parental. The rule is that the \emph{incoming branch} (the line that issues the potential reproduction event) is parental if it is of type $0$; otherwise, the \emph{continuing branch} (the target line on which the potential offspring is placed) is parental.  When all selective events have been resolved this way, the lines that are not parental are removed, and one is left with the {\em true genealogy} of the sample $\mathcal S$.

 Because of the positive recurrence (and the assumed time-stationarity) of the  line-counting process $(K_r)_{-\infty < r < \infty}$, there exists a.s. a sequence of (random) times $0 < t_1 < t_2< \ldots$ such that $t_n \to \infty$ and $K_{t_n} = 1$ for all $n$.
 Thus, for a given assignment of types to the lines of the stationary $\Lambda$-ASG $\mathcal A$ at time $0$, and for all $n \in \mathbb N$, removing the non-parental lines leaves exactly one true ancestral line, between the times $\underline{t}=0$ and $\overline{t}=t_n$, of the single individual in $\mathcal A$ at time $t_n$. The  resulting line between times  $\underline{t}=0$ and $\overline{t}=\infty$ is called the \textit{immortal line} or \textit{line of the common ancestor} in the  stationary $\Lambda$-ASG.
  
Our main result is a characterisation of its type distribution at time $0$, conditional on the type frequency in the population at that time. For the following definition, let $I_t$ be the type of the immortal line in the stationary $\Lambda$-ASG at time $t$.
\begin{definition}[Common ancestor type distribution]
	In the regime of Assumption \ref{assumptioncdi}, and for $x \in [0,1]$, let $h(x):=\mathbb{P}(I_{0}=0\mid X_0=x)$ be the probability that the immortal line in a stationary $\Lambda$-ASG with two-way mutations carries type $0$ at time~$0$, given the type-$0$ frequency in the population at time $0$ is $x$.
	\label{defhx} 
\end{definition}
By shifting the time interval $[0,t]$ back to $[-t,0]$, it becomes clear that $h(x)$  is also the limiting probability (as $t\to \infty$) that the ancestor at the past time $-t$ of the population at time~$0$ is of the beneficial type, given that the frequency of the beneficial type at time $-t$ was $x$.
\begin{theorem}\label{theoremhx}
	The probability $h(x)$ has the series representation
	\begin{equation}\label{gleichunghx}
	h(x)=\sum_{n\geq0} x(1-x)^{n}a_n,
	\end{equation}
	where the coefficients $a_n$ in \eqref{gleichunghx} are monotone decreasing, and the unique solution to the system of equations
	\begin{align}\notag
	& \sum_{n+1< c \le \infty} \left[\frac{1}{n}\binom{c-1}{c-n}\lambda_{c,c-n} \right] 
	(a_n-a_{c-1}) + (\sigma + \theta) a_n 
	=  \sigma a_{n-1} + \theta\nu_1 a_{n+1}, \quad n \geq 1,
	\\
	&\qquad  a_0=1, \quad a_{\infty}:=\lim_{n \to \infty} a_n=0 ,
	\label{thmrecursionanequation}
	\end{align} 
	with the convention 
	\begin{align}\label{convention}
	{\binom{\infty-1}{\infty-d+1}} := \begin{cases} 0 \quad \mbox{ if } d=1\\ 1 \quad \mbox{ if } d\ge 2 \end{cases},  \mbox{ and }\lambda_{\infty,\infty}:= \Lambda(\{1\}). 
	\end{align}
\end{theorem}
Let us discuss some special cases.
In the \emph{neutral case}, we clearly have $a_0=1$ and $a_n=0$ for $n > 0$, so $h(x)=x$, which is the neutral fixation
probability. For $\sigma>0$, we have $a_n > 0$ for all $n$, so $h(x)>x$ due to the higher-order terms in the series 
\eqref{gleichunghx}.
In the \emph{Kingman case}, 
the system of equations \eqref{thmrecursionanequation} simplifies to 
\begin{equation}
\left[ \frac{1}{n}\binom{n+1}{2} + \sigma + \theta \right] a_n 
= \frac{1}{n}\binom{n+1}{2} a_{n+1} + \sigma a_{n-1} + \theta\nu_1 a_{n+1}, \quad n \geq 1 ,
\label{thmrecursionanequationbinary}
\end{equation} 
and we immediately obtain	
\begin{corollary}[Fearnhead's recursion]
	In the Kingman case, the coefficients in \eqref{gleichunghx} satisfy the recursion
	\begin{equation}
	\left[\frac{1}{2}(n+1) + \sigma + \theta \right] a_n = \left[\frac{1}{2}(n+1) + \theta\nu_1\right] a_{n+1} + \sigma a_{n-1}, \quad n \geq 1,
	\label{fearnheadrecursion}
	\end{equation} 
	with $a_0=1$ and $\lim_{n\to \infty} a_n=0$.
	\label{corollaryfearnhead}
\end{corollary}	
The case $\Lambda(dz) = dz$, $0\le z \le 1$,  leads to the so called {\em Bolthausen-Sznitman coalescent}. Although the latter does not have the property CDI, we still have $\sigma^{\ast} = \infty$. In this case one has the identity  $\frac 1n{\binom{c-1}{c-n}}\lambda_{c,c-n}=\frac 1{(c-n-1)(c-n)}$ (cf. \cite{berestycki2009} Sec. 6.1), and the system \eqref{thmrecursionanequation} simplifies to
\begin{equation}
	\left[1 + \sigma + \theta \right] a_n =  \sigma a_{n-1}+  \theta\nu_1a_{n+1} +\sum_{j=1}^\infty \frac 1{j(j+1)} a_{n+j} , \quad n \geq 1,
	\label{bsrecursion}
	\end{equation} 
	with $a_0=1$ and $\lim_{n\to \infty} a_n=0$.

Recursion \eqref{fearnheadrecursion} appears in  \cite{Fearnhead} in connection with a time-stationary Wright-Fisher diffusion (with selection and mutation).\footnote{Note that there is a difference of a factor $1/2$ in the scaling of \eqref{fearnheadrecursion} in comparison to \cite{Fearnhead,Unserpaper,Taylor}. This is because these papers use the diffusion part of the Wright-Fisher generator (see \eqref{generatorlambdawf}) without the factor 1/2. This corresponds to a pair coalescence rate of 2 in the Kingman case, while in the present paper we assume pair coalescence rate 1 throughout.}  In \cite{Taylor}, the representation \eqref{gleichunghx} together with \eqref{fearnheadrecursion} was derived by analytic methods. In \cite{Unserpaper}, again for the Kingman case, we gave a new, more probabilistic proof, interpreting the coefficients $a_n$ as {\em equilibrium tail probabilities} of the {\em  line-counting process of the pruned lookdown ASG} (see Sec.~\ref{sec:pruned}).  In the present paper we give a twofold extension: (i) we include the case of multiple mergers, and (ii) we use a strong Siegmund duality (and thus a fully probabilistic method) in order to derive the recursion \eqref{thmrecursionanequation}.

An analogue of the quantity $h(x)$ can also be defined for a Moran model with finite population size $N$: for $k \in \{0,1,\ldots, N\}$,  let $h^{N}_k$ be the probability that the individual whose offspring will take over the whole population at some later time is of type $0$ at time $0$, given the number of type-$0$ individuals in the population at time $0$ is $k$. In  \cite{KluthHustedtBaake} it is shown (for the Kingman case) that $h^{N}_{k}$ converges to $h(x)$ as $N \to \infty$ and $k/N \to x$. Here, we work in the infinite-population limit right away, in order to carve out some important features of the underlying mathematical structure.


\section{\texorpdfstring{The pruned lookdown-$\Lambda$-ancestral selection graph}{The pruned lookdown-Lambda-ancestral selection graph}} \label{sec:pruned}

In the previous section, we have outlined the construction of the equilibrium $\Lambda$-ASG and layed out how the immortal line within it may be identified: Types are assigned
at time $0$, and the evolution is then followed forward in time.  In practice, however, this procedure is entangled due to the nested case distinctions required to identify the parental branch (incoming or continuing, depending on the type). In the Kingman case, we have solved this problem by \emph{ordering} the lines, and by \emph{pruning} certain lines
upon mutation \cite{Unserpaper}. The \emph{ordering} is achieved by arranging the coalescence events in a lookdown  manner, and by inserting the incoming branch
below the continuing branch at every selection event.  The \emph{pruning} takes care of the fact that the mutations convey information on the types of lines; this entails that some lines in the ASG can never be ancestral,  no matter which
types are assigned at time $0$, and can thus be deleted from the set of  potential ancestors. By construction, this removal  does not affect the
immortal line. 

More precisely, consider a realisation $\mathcal A$ of the ordered equilibrium ASG, decorated with the mutation events.  The corresponding lookdown version 
 is  obtained by placing the
lines on consecutive levels, starting at level $1$. We now proceed from $r=0$
in the direction of increasing $r$. When a  beneficial mutation event is encountered, we  delete all lines above it. When a deleterious mutation event occurs, we erase the
line that carries it; the lines above the affected line slide down to fill the space. One of the lines, called the {\em immune line}, is distinguished in that it is \emph{not} killed by mutations;
rather, it is relocated to the top. Let us anticipate that this is the line that is immortal if all lines at time 0 are assigned type 1. For illustrations and more details about the pruning procedure, see \cite{Unserpaper}. 

The resulting  \emph{pruned lookdown ASG} can also be generated
in one step, backward in time, in a Markovian manner. In what follows, we review this construction and extend it to 
the  pruned lookdown $\Lambda$-ASG. 

At each time $r$,  the pruned lookdown $\Lambda$-ASG $\mathcal G$ consists of a finite number $L_r$ of {\em lines}, 
i.e. the process $(L_r)_{r \in \mathbb{R}}$ takes values in the positive integers and  $L_r$ is the number of lines in $\mathcal G$ at time $r$.
The lines are numbered by the integers $1,\ldots, L_r$, to which we refer as {\em levels}. 
The evolution of the lines as $r$ increases is determined by a point configuration on $\mathbb{R} \times \big ( \mathcal{P}(\mathbb{N}) \cup (\mathbb N \times  \{\ast,\times,\circ \})\big)$, where $\mathcal{P}(\mathbb{N})$ is the set of subsets of $\mathbb{N}$ and $\mathcal P(\mathbb N)$ is equipped with  the $\sigma$-algebra generated by $\eta \mapsto \mathbf 1_\eta(i)$, $i \in \mathbb N$, $\eta \in \mathcal{P}(\mathbb{N})$.  Each of the points $(r, \tau)$ stands for a {\em transition element}~$\tau$ occurring at time $r$, that is, a  {\em merger}, a {\em selective branching}, a {\em deleterious mutation}, or a {\em  beneficial mutation} at time $r$. The level of the immune line at time $r$ is denoted by $M_r$; its precise meaning will emerge from Proposition \ref{LDASGstructure}.

Let us now detail the transition elements and their effects on $\mathcal G$ (see Figs.~\ref{flightselementsASG} and \ref{exampleASG}):

\begin{itemize} \item
	A {\em merger} at time $r$ is a pair $(r,\eta)$, where $\eta$ is a subset of $\mathbb N$. If \, $|\{1,\ldots, L_{r-}\}\cap~\eta|$ $ \le 1$, then $\mathcal G$ is not affected. If, however, $\{1,\ldots, L_{r-}\}\cap~\eta$ $ = \{j_1, \ldots, j_{\kappa}\}$ with $j_1 < \cdots < j_{\kappa}$ and $\kappa \ge 2$, then  the lines at levels $j_2, \ldots, j_\kappa$ merge into the line at level $j_1$. The remaining lines in $\mathcal G$ are relocated to `fill the gaps'
	while retaining their original order; this renders $L_r = L_{r-}-\kappa+1$. The immune line simply follows the line on level $M_{r-}$. 
	\item
	A {\em selective branching} at time $r$ is  a triple $(r,i,\ast)$, with  $i \in \mathbb N$. If $L_{r-} <  i$, then $\mathcal G$ is not affected. If $L_{r-} \ge i$, then a new line, namely the incoming branch, is inserted at level $i$ and all lines at levels $k\geq i$ (including the immune line if $M_{r-} \geq i$)
	are pushed one level upward to $k+1$, resulting in $L_r = L_{r-}+1$. In particular, the continuing branch is shifted from level $i$ 
	to  $i+1$. 
	\item 
	A {\em deleterious mutation}  at time $r$ is  a triple $(r,i,\times)$, with  $i \in \mathbb N$. If  $L_{r-} <  i$, then $\mathcal G$ is not affected. If $L_{r-} \ge i$ and $i \neq M_{r-}$, then the line at level $i$ is killed,  and the remaining lines in $\mathcal G$ (including the immune line) are relocated to `fill the gaps' (again in an order-preserving way), rendering $L_r = L_{r-}-1$. If, however, $i = M_{r-}$, then the line affected by the mutation is not killed but relocated to the currently highest level, i.e. $M_r= L_{r-}$. All  lines above  $i$ are shifted one level down, so that the gaps are filled, and in this case $L_r = L_{r-}$.
	\item 
	A {\em beneficial mutation}  at time $r$ is  a triple $(r,i,\circ)$, with  $i \in \mathbb N$. If  $L_{r-}  < i$, then $\mathcal G$ is not affected.  If $L_{r-} \geq i$, then all the lines at levels $> i$ are killed, rendering $L_r=i$,  and the immune line is relocated to $M_r = i$.
\end{itemize}
\begin{proposition} \label{LDASGstructure}
	Assume that for some $r_0 < 0$ we have $L_{r_0} = 1$, and assume there are finitely many transition elements that affect $\mathcal G$ between times $r_0$ and $0$. Consider an arbitrary assignment of types to the $L_0$ lines at time $r=0$.  Then  the level of the immortal line at time $0$ is either the lowest type-$0$ level at time $0$ or, if all lines at time $0$ are of type $1$, it is the level $M_0$ of the immune line at time $0$. In particular, the immortal line is of type $1$ at time $0$ if and only if all lines in $\mathcal{G}$ at time $0$ are assigned type $1$.
\end{proposition}
\begin{proof}
In the absence of multiple mergers (i.e. if all mergers have exactly two elements), this is Theorem 4 in \cite{Unserpaper}. In its proof, the induction step for binary mergers  directly carries over to multiple mergers. 
\end{proof}

\begin{figure}[h]
	\centering
		\begin{minipage}[t]{0.95\textwidth}
			\centering
			\vspace{-0.1cm}
			\includegraphics[width=0.9\textwidth]{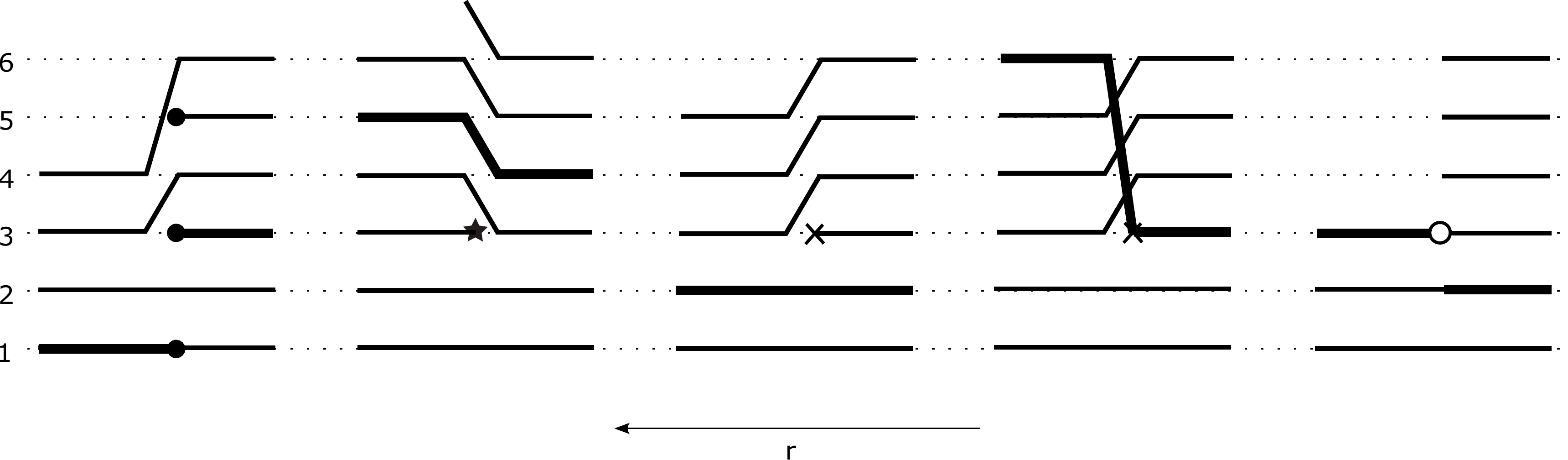}
			\caption{\small{Transitions  of the pruned lookdown {$\Lambda$-ASG}. Since the graph evolves `into the past',   time $r$ runs from right to left in the figure. 
			The value of $L$ is 6 before the jump; the immune line is marked in bold. From left to right: A `merger' of the lines on levels $1$, $3$, and $5$ (indicated by bullets);   a `star' at level $3$; a `cross' at level~$3$, outside the immune line; a `cross' on the immune line at level $3$; a `circle' at level $3$.    }}
			\label{flightselementsASG}
		\end{minipage}
\end{figure}
\begin{figure}[h]
	\centering
		\begin{minipage}[t]{0.95\textwidth}
			\centering
			\vspace{-0.1cm}
			\includegraphics[width=0.9\textwidth]{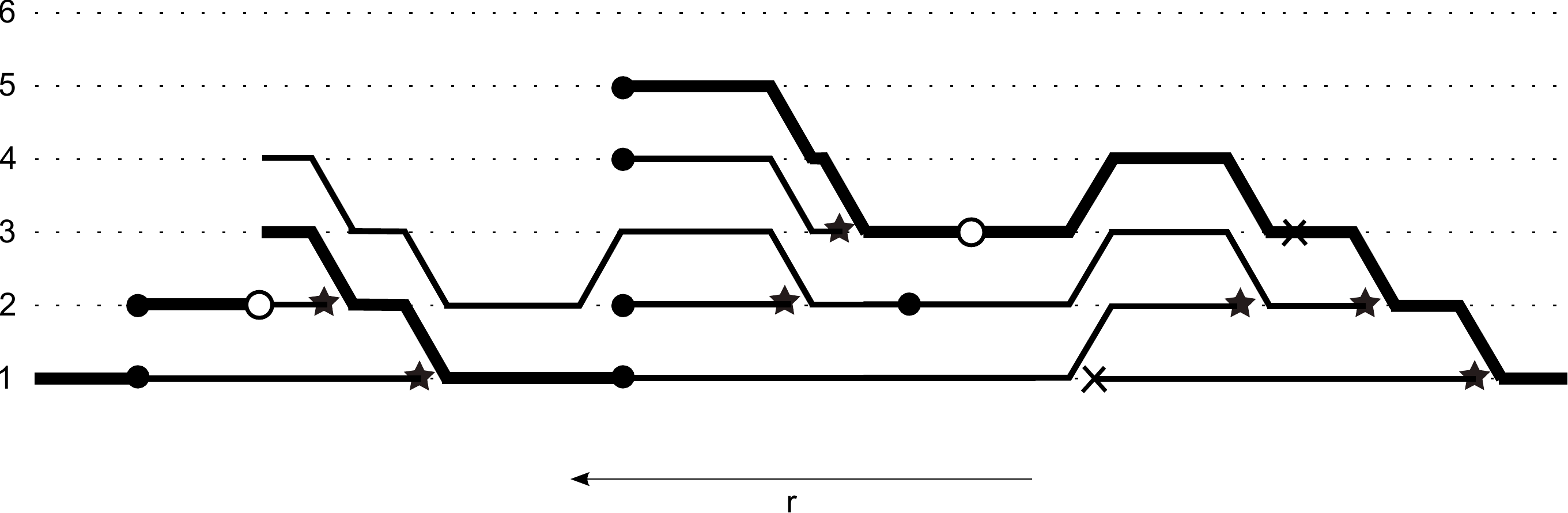}
			\caption{\small{A cut-out of a realisation of the pruned lookdown $\Lambda$-ASG. The immune line is marked in bold.}}
			\label{exampleASG}
		\end{minipage}
\end{figure}

Taking together the  above description of $\mathcal G$  and the rates defining the $\Lambda$-ASG (Sec.~2), we can
now summarise and formalise the law of $\mathcal G$ as follows.
The transition elements arrive via independent Poisson processes:
For each $i \in \mathbb N$,  the `stars', `crosses', and `circles' at level $i$ come as  Poisson processes with intensities $\sigma$, $\theta \nu_1$ and $\theta \nu_0$, respectively. For each 2-element subset $\eta$ of $\mathbb N$, the `$\eta$-mergers' come as a Poisson process with intensity $\Lambda(\{0\})$. In addition, we have a Poisson process with intensity measure $\mathbf{1}_{\{z>0\}}\frac 1{z^2}\Lambda(dz) \, dr $, where each $z$ generates a random subset $H(z):= \{i : V_i =1\} \subset \mathbb N$, with $(V_i)_{i \in \mathbb N}$ being a Bernoulli$(z)$-sequence, and the point $(r,z)$ gives rise to the merger $(r,H(z))$.  All these Poisson processes are independent. The points $(r, \tau)$ constitute a Poisson configuration $\Psi$, whose intensity measure we denote by $\mu \otimes \rho$, where   $\mu$ is  Lebesgue measure on $\mathbb R$.
With the transition rules described above, this induces Markovian jump rates upon $L_r$ and $(L_r,M_r)$. With the help of \eqref{mergerrateofkblocks}, it is easily checked that the generator $G_L$ of $L$ is  given by 
\begin{equation}
\begin{split}
G_{L} g(\ell) = & \sum_{c=1}^{\ell-1}{\binom{\ell}{\ell-c+1}} \lambda_{{\ell},{\ell-c+1}} \left[g(c)-g(\ell)\right] + \ell\sigma \left[g(\ell+1)-g(\ell)\right]\\
& + (\ell-1)\theta\nu_1 \left[g(\ell-1)-g(\ell)\right] + \sum_{k=1}^{\ell-1}\theta\nu_0 \left[g(\ell-k)-g(\ell)\right].
\end{split}
\label{generatorprunedlambdaldasg}
\end{equation}

Due to Assumption \ref{assumptioncdi} and Remark \ref{recurrence}b), and because $L$ is stochastically dominated by $K$, the process $L$ obeys
\begin{align}\label{hitzero}
\mathbf E_{\ell} [T_1] < \infty, \qquad \ell \in \mathbb N.
\end{align}
Thus $L$ has a time-stationary version $\widetilde L$ (which is $\widetilde L \equiv 1$ if $\sigma =0$), and likewise the pruned lookdown $\Lambda$-ASG has an equilibrium version as well. We now set $L_{\text{eq}} :=\widetilde L_0$ and denote the  tail probabilities of $L_{\text{eq}}$ by 
\begin{equation}
\alpha_n:=\mathbb{P}(L_{\text{eq}}>n), \quad n \in \mathbb{N}_0.
\label{deftailprobabilities}
\end{equation}
Because of \eqref{hitzero}, for almost all realisations of $\widetilde L$, there exists an $r_0<0$ such that $\widetilde L_{r_0} = 1$. Hence, arguing as in \cite[proof of Theorem 5]{Unserpaper}, we conclude from Proposition \ref{LDASGstructure} the following
\begin{corollary}
	Given the frequency of the beneficial type at time $0$ is $x$, the probability that the immortal line in the equilibrium p-LD-$\Lambda$-ASG at time $0$ is of beneficial type is
	\begin{equation}	\label{hrep}
	h(x)=\sum_{n\geq0} x(1-x)^{n}\alpha_n.
	\end{equation}
\end{corollary}
In order to further evaluate the representation \eqref{hrep}, we need information about the equilibrium tail probabilities $\alpha_n$. This is achieved in the following sections via  a process $D$ which is a Siegmund dual for $L$. 

\section{An application of Siegmund duality}
\label{sectionsiegmunddual}

A central point in our proof of Theorem \ref{theoremhx} will be that the equilibrium tail probabilities
of $L$ can be expressed as certain hitting probabilities of a process $D$ which is a so-called Siegmund dual of $L$. The relationship between the transition semigroups of $L$ and $D$ is given by formula \eqref{dualitycliffordsudbury} below. Intuitively, the process $D$ may be seen as going into the opposite time direction as $L$. In a suitable representation via stochastic flows, which turns out to be available for monotone processes, \eqref{dualitycliffordsudbury} means that the paths of $D$ remain `just above' those of $L$, see Sec. \ref{sectionflights} below.

\subsection{Tail probabilities and hitting probabilities}
\label{sectioncoxroesler}

It is clear that $L$ is stochastically monotone, that is, $\mathbb{P}_n(L_r\geq i) \geq \mathbb{P}_m(L_r\geq i)$ for $n \geq m$ and for all $i \in S$ (where the subscript refers to the initial value of the process). It is well known \cite{Siegmund} that such a process has a Siegmund dual, that is, there exists a process $D$ such that
\begin{equation}
\mathbb{P}_{\ell}(L_u \geq d) = \mathbb{P}_d(D_u \leq \ell) 
\label{dualitycliffordsudbury}
\end{equation} 
for all $u \geq 0$, $\ell,d \in \mathbb N$.

\begin{lemma}\label{tailhit}
	The tail probabilities of the stationary distribution of $L$ are hitting probabilities of the dual process $D$. To be specific,
	\begin{equation}
	\alpha_n=\mathbb{P}_{n+1}(\exists t\geq 0 : D_t=1) \quad \forall n \geq 0.
	\label{thmtailprobabilitiesequation}
	\end{equation} 
	\label{theoremCoxRoesler}
\end{lemma}
\begin{proof}
This is a special case of  \cite[Thm.~1]{CoxRoesler} for entrance and exit laws. In our case the entrance law is the equilibrium distribution of $L$, the exit law is a harmonic function (in terms of hitting probabilities), and the proof reduces to the following elementary argument. Namely,
evaluating
the  duality condition \eqref{dualitycliffordsudbury} for $\ell=1$ and $d=n+1$, $n\geq 0$, gives 
\begin{equation}
\mathbb{P}_{1}(L_u\geq n+1) = \mathbb{P}_{n+1}(D_u=1) \quad \text{for all } u \geq 0, \ n\geq 0.
\label{proofSiegmund1}
\end{equation} 
Taking the limit $u \to \infty$, the left-hand side converges  to  $\mathbb{P}(L_{\text{eq}}> n)= \alpha_n$ by positive recurrence and irreducibility. Setting $\ell=d=1$ in \eqref{dualitycliffordsudbury}, we see that $1$ is an absorbing state for $D$. Hence we have for the right-hand side of \eqref{proofSiegmund1}
\begin{equation*}
\lim_{u \to \infty} \mathbb{P}_{n+1}(D_u=1)= \mathbb{P}_{n+1}(\exists t\geq 0 : D_t=1) \quad \forall n \geq 0,
\label{proofSiegmund2}
\end{equation*} 
and the lemma is proven.
\end{proof}

Next we want to show that the (shifted) hitting probabilities 
\begin{equation}\label{hitprob}
\alpha_n = \mathbb{P}_{n+1}(\exists t\geq 0 : D_t=1), \quad  n \geq 0,
\end{equation}
satisfy the system of equations \eqref{thmrecursionanequation}. More precisely,   \eqref{thmrecursionanequation} will emerge as a first-step decomposition of the hitting probabilities. For this purpose, we first have to identify the jump rates of~$D$. This can be done  via a generator approach that translates the jump rates of the process $L$ (which appear in \eqref{generatorprunedlambdaldasg}) into their dual jump rates, see, for instance,  formula (12)  in \cite{CliffordSudbury} or in \cite{Siegmund}. For the jump rates coming from the mergers this is somewhat technical, see the calculations in the appendix  in \cite{olivier}.

Inspired by  \cite{CliffordSudbury} we will therefore take a `strong pathwise approach' that consists in decomposing the dynamics of $L$ into so-called \emph{flights}, which can be `dualised' one by one. While Clifford and Sudbury, starting from the generator of a monotone process, in \cite[Thm 1]{CliffordSudbury} construct a special Poisson process of flights  for which they form the duals (\cite[Thm 2]{CliffordSudbury}), in our situation the Poisson process of flights is naturally given (being induced by  the transition elements for $\mathcal G$ defined in Sec.~\ref{sec:pruned}, see Sec.~\ref{sec:dual_for_L} below). Consequently, we will show in Proposition~\ref{propdual} that the approach of \cite[Thm 2]{CliffordSudbury} works also when starting from a more general Poisson process of flights.

\subsection{Flights and their duals}
\label{sectionflights}
In \cite{CliffordSudbury}, Clifford and Sudbury introduced a graphical representation that allows us to construct a monotone homogeneous Markov process $\mathcal L$ together with its Siegmund dual $\mathcal D$ on one and the same probability space. The method requires that the state space $S$ of the processes $\mathcal L$ and $\mathcal D$ is  (totally) ordered. We restrict ourselves to the case $S:=\mathbb{N}\cup\{\infty\}$, which is the relevant one in our context (and which is prominent in \cite{CliffordSudbury} as well).

The basic  building blocks of Clifford and Sudbury's construction are so-called  \emph{flights}. A flight~$f$ is a mapping from $S$ into itself that is order-preserving, so $f(k) \leq f(\ell)$ for all $k < \ell$ with $k,\ell  \in  S$;  let us add that each flight leaves state $\infty$ invariant, so $f(\infty) = \infty$.
By the construction described below,  a flight $f$ that appears at time $r$ will induce the  transition to $\mathcal{L}_r=f(\ell)$, given $\mathcal{L}_{r-} =\ell$.  This way, transitions from different initial states will be coupled on the same probability space. A flight $f$ is  graphically represented as a set of simultaneous arrows pointing from $\ell$ to $f(\ell)$, for all $\ell \in S$, so that the process simply follows the arrows. Examples are shown in Fig. \ref{flightsaloneL}.

We denote the set of all flights by $\mathcal F$, and consider a Poisson process $\Phi$ on $\mathbb R \times \mathcal F$ whose intensity measure is of the form $\mu \otimes \gamma$, where $\mu$ is  again Lebesgue measure on $\mathbb R$,  and the measure $\gamma$ has the property
\begin{align}\label{locfin}
\gamma(\{f\in \mathcal F: f(\ell)\neq \ell\}) < \infty, \qquad \ell \in \mathbb N. 
\end{align}
Property \eqref{locfin} implies that with probability $1$, for all $\ell \in \mathbb N$ and $r\in \mathbb R$, among all the points $(s,f)$ in $\Phi$ with $s> r$ and $f(\ell) \neq \ell$, there is one whose $s$ is minimal. We denote this time by  $v(r, \ell)$. For $r \in \mathbb R$ and $\ell \in \mathbb N$,  we  define inductively a sequence $(s_0,\ell_0), (s_1,\ell_1), \ldots$ with $r=: s_0 <s_1<\cdots$, $\ell=:\ell_0, \ell_1,\ell_2,\ldots \in S$, by setting $s_i:= v(s_{i-1},\ell_{i-1})$, $\ell_i := f(\ell_{i-1})$, with $(s_i, f) \in \Phi$. (Note this procedure will terminate if  $\ell_i =\infty$ for some $i \in \mathbb N$.)

With the notation just introduced,  $\Phi$  induces a semi-group (a {\em flow}) of mappings, indexed by $r<s \in \mathbb R$, and defined by
\begin{align}\label{flow}
F_{r,s}(\ell):=\begin{cases} \ell_i \quad \ \mbox{ if } s_i \le s < s_{i+1}, \\ \infty \quad  \mbox{ if } \lim_{i \to \infty} s_i \le s \end{cases}
\end{align}
for $\ell \in \mathbb N$, with $F_{r,s}(\infty):= \infty$.

Assuming property \eqref{locfin}, we say that  $\Phi$ {\em represents} the  process $\mathcal L$ if  for all $s>0$ the  distribution of $F_{0,s}(\ell)$ is a version of the conditional distribution of $\mathcal L_s$ given $\{\mathcal L_0=\ell\}$, $\ell \in \mathbb N$. Equivalently, for all $r \in \mathbb R$ and $u > 0$,
\begin{equation} \label{repr}
\mathbb{P}_{\ell}(\mathcal L_u \in (.)) = \mathbb{P}(F_{r,r+u}(\ell) \in (.)).
\end{equation}

We now describe, in the footsteps of Clifford and Sudbury \cite{CliffordSudbury}, the construction of  a strong pathwise Siegmund dual $\mathcal D$, based on the same realisation of the flights as for the original process $\mathcal L$.  Def. \ref{dualflights} $a)$ formalises the statement at the beginning of Sec. \ref{sectionsiegmunddual} that the paths of $D$ remain `just above' those of $L$, see also Fig. \ref{flightsaloneL} for an illustration.

\begin{definition}[Dual flights] a) For a flight $f: S\to S$, its {\em dual flight} $\widehat f$ is defined by
	\begin{equation}
	\widehat f(d)=\min(f^{-1}(\{d, d+1,  \ldots\})), \quad  d \in S,
	\label{dualequation}
	\end{equation}
	with the convention $\min ( \varnothing ) =\infty$. 
	\label{dualflights}
	
	b) For a Poisson process $\Phi$ on $\mathbb R \times \mathcal F$, we define $\widehat {\Phi}$ as the result of $\Phi$ under the mapping $(r,f) \mapsto (-r, \widehat f)=:(t,\widehat f)$. Moreover, under the assumption
	\begin{equation}\label{duallocfin2}
	\gamma(\{f\in \mathcal F: \widehat f(d)\neq d\}) < \infty, \quad  d \in \mathbb N, 
	\end{equation}	
	we define $\widehat F$ in terms of $\widehat \Phi$ in the same way as $F$ was defined in terms of $\Phi$ by \eqref{flow}.
\end{definition}

It is clear that $\widehat f$ is order preserving. Since $f$ is monotone increasing by assumption, we have 
$\max(f^{-1}(\{1, \ldots, d-1\})) \leq \min(f^{-1}(\{d, d+1,  \ldots\}))$. 
Because $f^{-1}(\{1, \ldots, d-1\}) \cap f^{-1}(\{d, d+1,  \ldots\}) = \varnothing$  and  $f^{-1}(\{1, \ldots, d-1\}) \cup f^{-1}(\{d, d+1,  \ldots\}) = S$, we see that \eqref{dualequation} is equivalent to
	\begin{equation}
	\widehat f(d)=\max(f^{-1}(\{1, \ldots, d-1\}))+1, \quad  d \in S, 
	\label{dualequation2}
	\end{equation}
	with the convention $\max ( \varnothing ) =0$. 
Note further  that \eqref{duallocfin2} is implied by \eqref{locfin} together with 
\begin{align}\label{duallocfin}
\gamma(\{f\in \mathcal F: \exists k > \ell \text{ s.t. }  f(k)\leq \ell\}) < \infty, \qquad \ell \in \mathbb N. 
\end{align}

The following proposition is an adaptation of \cite[Theorem~2]{CliffordSudbury} to our setting. Compare also \cite[Section 4.1]{JansenKurt}.
\begin{proposition}\label{propdual}
	Assume \eqref{locfin} and \eqref{duallocfin}, and assume that $\infty$ is unattainable for the process  $\mathcal L$ represented by the Poisson process $\Phi$ with intensity measure $\mu \otimes \gamma$. Then the following strong pathwise duality relation is valid:  For all $s > 0; \, \ell, d \in \mathbb N$,
		\begin{align}\label{stronglypathwiseduality}
	\mathbf 1_{\{F_{-s,0}(\ell) \ge d\}} = \mathbf 1_{\{\widehat F_{0,s}(d) \le \ell\}}, \qquad \mbox{ almost surely.}
	\end{align}
\end{proposition}

\begin{proof}
Let   $Y:= (Y_r)_{r \in [-s,0]} := (F_{-s,r}(\ell))_{r \in [-s,0]}$, and $\widehat Y:= (\widehat Y_t)_{t \in [0,s]} := (\widehat F_{0,t}(d))_{t \in [0,s]}$, for given $\ell$, $d$, and $s$. Due to \eqref{locfin} and the assumption that $\infty$ is unattainable, $Y$ has a.s only finitely many jumps; let us denote the jump times by  $-r_1, \ldots, -r_n$. We write $\widehat J$ for the union of $\{r_1, \ldots, r_n\}$ and the set of jump times  of $\widehat Y$. Because of \eqref{duallocfin2}, $\widehat J$ has a smallest element, a second-smallest element, and so on. We denote these elements by $u_1 < u_2 < \ldots$, and show that 
\begin{equation}\label{notraverse}  Y_{0} \ge \widehat Y_{0}  \mbox{ if and only if }  Y_{(-u_i)-} \ge \widehat Y_{u_i}, \qquad i=1,2,\ldots
\end{equation}
Proceeding by induction, for \eqref{notraverse} it is sufficient to show
\begin{equation}
\mathbf{1}_{\{f(j) \geq k\}} =	\mathbf{1}_{\{\widehat f(k) \le j\}}  
\label{stronglypathwisedualityinf}
\end{equation}
for all flights $f$, and $j, k \in \mathbb N$.
Let $f\in \mathcal F$.  On the one hand, 
$f(j) \geq k$ yields  
\[
\widehat f (k) \leq \widehat f \bigl (f(j)\bigr ) =  \min \bigl (f^{-1}(\{f(j), f(j)+1, \ldots\})\bigr) = \min \bigl (f^{-1}(f(j))\bigr ) \leq j,
\]
where we have used order preservation of $\widehat f$ and $f$ as well as \eqref{dualequation}.
On the other hand, $f(j) < k$ is equivalent to $f(j)+1 \leq k$. By order preservation and \eqref{dualequation2}, this entails 
\[
\widehat f(k) \geq \widehat f \bigl (f(j+1)\bigr ) = \max \bigl (f^{-1}(\{1, \ldots, f(j)\})\bigr )+1  =  \max \bigl (f^{-1}(f(j))\bigr )+1  \geq j+1 > j.
\]
We have thus shown \eqref{stronglypathwisedualityinf}, and hence also \eqref{notraverse}.

If $(u_i)$ has no accumulation point, then it has a maximal element, say $u_m$. Choosing $i=m$ in the r.h.s. of \eqref{notraverse} yields \eqref {stronglypathwiseduality} (since $u_m \neq s$ with probability 1). If $(u_i)$ has an accumulation point, say  $\tau$, then, because of \eqref{duallocfin2}, we have \,  $\lim_{t \uparrow \tau} \widehat Y_{t}=\lim_{i\to \infty} \widehat Y_{u_i}$ $ = \infty$. Because $Y$ remains bounded by assumption, this together with \eqref{notraverse} enforces that $Y_0 < \widehat Y_{0}$. This means that the l.h.s. of \eqref{stronglypathwiseduality} takes the value 0. However, this is the case also for the r.h.s of \eqref{stronglypathwiseduality}, since $\infty = \widehat Y_{\tau} = \widehat Y_s > \ell$.
\end{proof}

In view of \eqref{repr} we  immediately  obtain the following 
\begin{corollary}\label{dualcoro}
	In the situation of Proposition \ref{propdual}, let $\mathcal D$ be a process represented by $\widehat \Phi$. Then $\mathcal L$ and $\mathcal D$ satisfy the duality relation \eqref{dualitycliffordsudbury}, with $L$ and $D$ replaced by $\mathcal L$ and $\mathcal D$.
\end{corollary}

\subsection{\texorpdfstring{A Siegmund dual for the process $L$}{A Siegmund dual for the process L}} \label{sec:dual_for_L}
Let us now turn to our case where $\mathcal{L}=L$. With each of the transition elements $\eta$, $(i, \ast)$, $(i, \times)$, $(i, \circ)$ introduced in Sec.~\ref{sec:pruned} 
we associate a flight defined as follows ($\ell \in S, i \in \mathbb N$):
\begin{equation}\label{flightsL1}
\begin{split}
& f_{\eta}(\ell) =  \ell - | \{ 1, ...., \ell \} \cap \tilde \eta |, \qquad \text{where } \tilde \eta := \eta \setminus \{\min (\eta)\}, \\[2mm]
& f_{i, *}(\ell) = \begin{cases}  \ell, & \ell < i, \\  \ell+1, & \ell \geq i, \end{cases}  \qquad
f_{i,\times}(\ell) = \begin{cases} \ell, & \ell \leq i, \\  \ell-1, & \ell > i,  \end{cases} \qquad 
 f_{i, \circ}(\ell) = \begin{cases}  \ell, & \ell \leq i, \\  i, & \ell > i, \end{cases}
\end{split}
\end{equation}
compare also Fig. \ref{flightsaloneL}. The flights are  indeed  order preserving. The structure of $f_{\eta}$, $f_{i, *}$, and $f_{i, \circ}$ is clearly inherited from that of the corresponding transition elements. The flights $f_{i,\times}(\ell)$ forget about the position (but not about the existence) of the immune line within the p-LD-$\Lambda$-ASG.  Indeed, recall that the downward jump rate of $L$  due to deleterious mutations is  $(\ell - 1) \theta \nu_1$; this reflects the fact that crosses arrive at rate $\theta \nu_1$ per line, but are ignored on the immune line, no matter where it is located. This is taken into account in the definition of the flight  $ f_{i,\times}$ by setting $f_{\ell,\times}(\ell) = \ell$. 

Let us now start from the Poisson configuration $\Psi$ (of points $(r,\tau)$ with intensity measure $\mu \otimes \rho$), as described in Sec. \ref{sec:pruned}. Let $\gamma$ be the image of the measure $\rho$ under the mapping $\tau \mapsto f_\tau$, where $f_\tau$ is the flight belonging to the transition element $\tau$ as defined in \eqref{flightsL1}. The measure $\gamma$ has property \eqref{locfin}. To see this we write  $\gamma= \gamma_m + \gamma_\ast + \gamma_\times + \gamma_\circ$, where the 4 summands describe the intensity measures of the flights stemming from the mergers, the selective branchings, the deleterious mutations and the beneficial mutations. It is straightforward that  $\gamma_\ast, \gamma_\times$ and $\gamma_\circ$ obey \eqref{locfin}. To see that also  $\gamma_m$ obeys \eqref{locfin}, note that for $\ell \in \mathbb N$
\begin{align}\label{locfinex}
\gamma_m(\{f\in \mathcal F: f(\ell)\neq \ell\}) = \rho(\{\eta: |\eta\cap \{1,\ldots,\ell\} |\ge 2\})\le {\binom{\ell}{2}},
\end{align}
since $\{\eta: |\eta\cap \{1,\ldots,\ell\} |\ge 2\} \subset \bigcup_{1\le i < j \le \ell}\{\eta:\{i,j\} \subset \eta\}$ and because for all $i < j \in \mathbb N$
\begin{align}\label{mergingrates}
\rho(\{\eta:\{i,j\} \subset \eta\}) = \int_{(0,1]} z^2 \frac 1{z^2} \Lambda(dz) + \Lambda(\{0\})=1.
\end{align}
 
Writing $\Phi$  for the Poisson point process with 
 intensity measure $\mu \otimes \gamma$, it is now clear that  $\Phi$ represents the process~$L$ in the sense of \eqref{flow} and \eqref{repr}, because the jump rates match those appearing in the generator \eqref{generatorprunedlambdaldasg}.

\begin{figure}[htbp]
	\centering
		\begin{minipage}[t]{0.9\textwidth}
			\centering
			\vspace{-0.1cm}
			\includegraphics[width=\textwidth]{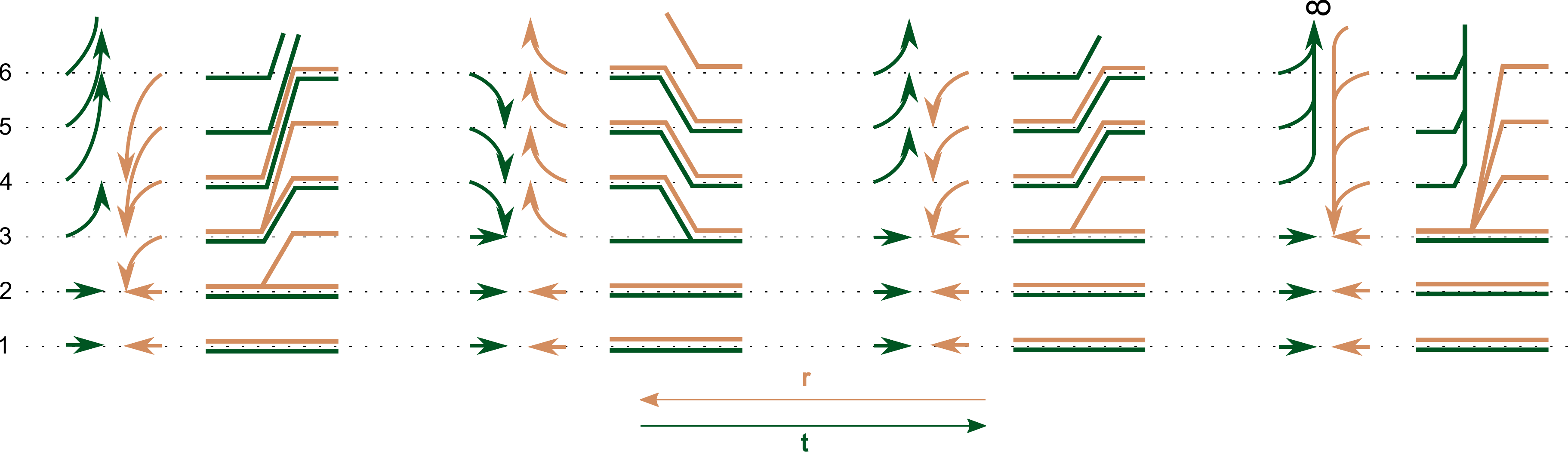}
			\caption{\small{Graphical representation of the four types of flights defined in \eqref{flightsL1} (light brown arrows) and their dual flights as defined in \eqref{flightsD} (dark green arrows), together with the resulting paths of $L$ (light brown)  and  $D$ (dark green). The flights displayed  are $f_{\eta}$ (with $\eta \cap \{1, \ldots,6\} = \{1,3,5\}$), $f_{3, *}$, $f_{3,\times}$, $f_{3,\circ}$; and $\widehat{f}_{\eta}$, $\widehat{f}_{3, *}$, $\widehat{f}_{3,\times}$, $\widehat{f}_{3,\circ}$. The flight $\widehat{f}_{3,\circ}$ maps all states $d > 3$ to the absorbing state $\infty$. The paths of $L$ and $D$  follow the
			arrows in the direction of backward and forward time, respectively.}}			
		\label{flightsaloneL}
		\end{minipage}
\end{figure}

Let us now check that $\gamma$ also satisfies assumption \eqref{duallocfin}.  It is straightforward that  $\gamma_\ast, \gamma_\times$ and $\gamma_\circ$ obey \eqref{duallocfin}. To see that also  $\gamma_m$ obeys \eqref{duallocfin}, we note that for $k \ge 2\ell +2$ and $\eta \subset \mathbb N$ the inequality $f_\eta(k)  \le \ell $ implies that  $\lvert \eta \cap \{1,\ldots, \ell+1\}\rvert \, \ge 1$ and  $\lvert \eta \cap \{\ell+2,\ldots, 2\ell+2\}\rvert \,  \ge 1$. Let $H_\ell$ denote the set of all $\eta \in \mathcal P(\mathbb N)$ having the latter property. Then we have for all $\ell \in \mathbb N$ the estimate
\begin{align*}
&\gamma_m(\{f\in \mathcal F: \exists k > \ell \text{ s.t. }  f(k)\leq \ell\}) \\ &\le \sum_{k=\ell+1}^{2\ell +1}\gamma_m (\{f\in \mathcal F: f(k)\neq k\})
+\gamma_m(\{f\in \mathcal F: \exists k \ge 2\ell +2 \text{ s.t. }  f(k)\leq \ell\}) 
\\
& \le \sum_{k=\ell+1}^{2\ell +1} {\binom{k}{2}}+ \rho(H_\ell) <\infty,
\end{align*}
because of \eqref{locfinex} and \eqref{mergingrates}, since $H_\ell \subset \bigcup_{1\le i \le  \ell+1 < j \le 2\ell+2} \{\eta: \{i,j\}\subset \eta\}$.

Following Definition \ref{dualflights}, we can now consider a process $D$ represented by $\widehat \Phi$. According to Corollary \ref {dualcoro}, $L$ and $D$ then obey the duality relation \eqref{dualitycliffordsudbury}. It remains to read off the jump rates of $D$ from the intensities of the (dual) flights.
\begin{lemma}\label{GenD}
	The generator $G_D$ of the process $D$ is given by
	\begin{equation}
	\begin{split}
	G_{D} g(d) = &\sum_{d<c\le \infty}{\binom{c-1}{c-d+1}} \lambda_{{c},{c-d+1}} \left[g(c) - g(d)\right] +(d-1)\sigma\left[g(d-1)-g(d)\right]\\
	& + (d-1)\theta\nu_1 \left[g(d+1)-g(d)\right] + (d-1)\theta\nu_0 \left[g(\infty)-g(d)\right] \ , \quad d \in \mathbb{N}, g: S \to \mathbb{R}.
	\label{generatordualprozess}
	\end{split}
	\end{equation}
	\label{dualgeneratorlemma}
	where we again use the convention \eqref{convention}.
	\end{lemma}

\begin{proof}
We claim that the flights that are dual to those in \eqref{flightsL1} are of the form
\begin{equation}\label{hatfeta}
\widehat f_{\eta}(d) =  \min\{\ell:|\{1,\ldots,\ell\}\cap (\mathbb N\setminus \tilde \eta)| = d \}, \qquad 
\mbox{again with } \tilde \eta = \eta \setminus \{\min (\eta)\},  
\end{equation}
\begin{equation}\label{flightsD}
\widehat f_{i, *}(d) = \begin{cases}  d, & d \leq i, \\  d-1, & d > i \end{cases} \qquad
\widehat f_{i,\times}(d) = \begin{cases}  d, & d \leq i,  \\  d+1, & d > i \end{cases} \qquad 
\widehat f_{i, \circ}(d) = \begin{cases}  d, & d \leq i,  \\  \infty, & d > i, \end{cases}
\end{equation}
$d \in S$, $i \in \mathbb{N}$ (see   Fig. \ref{flightsaloneL}).

The
expressions in \eqref{flightsD} are obvious consequences of \eqref{dualequation} and \eqref{flightsL1}.
To verify \eqref{hatfeta}, we first note that, due to  Definition  \ref{dualflights}, we have   
$\widehat f_\eta(d)=\min\bigl ( f^{-1}_\eta(d) )$, since $f_\eta$ is surjective and monotone increasing. Consequently, in the case  $|\{1, \ldots,d\} \cap \eta| \leq 1$ we have $\widehat f_{\eta}(d)=d$, whereas otherwise we have
$\widehat f_\eta (d)  = \min \{ \ell: | \{1, ...., \ell \} \cap \tilde \eta | = \ell - d  \} > d$, both in accordance with \eqref{hatfeta}.

Let us now consider the contribution of the various types of flights to $G_D$. For $c\neq d \in \mathbb N$ we have to compute $\gamma(\{f: \widehat f(d) =c\})$. It is clear that the contributions from $\gamma_\ast$, $\gamma_\times$ and $\gamma_\circ$ yield the last 3~summands in 
\eqref{generatordualprozess}. For the contribution coming from $\gamma_m$, we have for $d < c < \infty$
\begin{align}\label{intmerge}
\gamma_m(\{f: \widehat f(d) =c\}) = \rho(\{\eta: c \notin \eta,  \lvert \{1, \ldots,c-1\} \cap \eta \rvert = c-d+1\}).
\end{align}
The contribution from the Kingman mergers to the right-hand side of of \eqref{intmerge} is $\Lambda(\{0\}) \binom{c-1}{2}$ if $c=d+1$, and $0$ otherwise.
For $z>0$, the probability that a $z$-merger does not affect level $c$ but does affect $c-d+1$ out of the levels $1,\ldots,c-1$ is  $\binom{c-1}{c-d+1}z^{c-d+1} (1-z)^{d-1} $. Integrating this  with respect to  $\frac 1{z^2}\Lambda(dz)$ and adding the Kingman component shows that the right-hand side of \eqref{intmerge} equals ${\binom{c-1}{c-d+1}} \lambda_{{c},{c-d+1}}$. These are the jump rates from $d$ to $c< \infty$ that appear in the first sum on the r.h.s. of \eqref{generatordualprozess}. It remains to take into account the jump rate of $D$ from $d$ to $\infty$. For this we note that $f_{\mathbb N}(\ell) = 1$, $\ell = 1,2,\ldots$, and consequently $\widehat f(d) = 1$ if $d=1$ and $\widehat f(d) = \infty$ if $d \ge 2$. These flights appear at rate $\Lambda(\{1\})$, and thus for $d\ge 2$ add the term $(g(\infty)-g(d))\Lambda(\{1\})$ to the generator.
\end{proof}

\begin{remark} \label{connectionfixationline}
	In the case without selection and mutation (that is, $\sigma=\theta=0$), our process $D$ shifted by one, that is, $D-1$, is equal to the so-called {\em fixation line} in \cite{olivier}. In this case one has no pruning, and the line-counting process $K$ has generator \eqref{generatorlambdaasg} (with $\sigma = 0$).  The (Siegmund) duality between $K$ and $D$ is stated in \cite[Lemma 2.4]{olivier}. See also  \cite[Thm 2.3]{GaiserMoehle} for a corresponding statement on the still more general class of exchangeable coalescents.
\end{remark}

We now come  to the
\begin{proof}[Proof of Theorem~\ref{theoremhx}.]
Consider the tail probabilities $\alpha_n = \mathbb P(L_{eq} > n)$, $n \in \mathbb N_0$, as defined in \eqref{dualitycliffordsudbury}. Lemma \ref{dualitycliffordsudbury} allows us to write them as hitting probabilities of $D$. Specifically, with 
$$\omega(n):= \mathbb P_n(\exists t \ge 0: D_t=1),$$ 
we have $\omega(n) = \alpha_{n-1}$. 
The hitting probabilities $\omega(n)$, $2\le n < \infty$, constitute a $G_D$-harmonic function, that is,
\begin{equation}\label{harmonic}
G_D \omega(n) = 0,\quad n\ge 2.
\end{equation}
It is this relation that is equivalent to the system \eqref{thmrecursionanequation}.
Indeed, \eqref{harmonic} translates into the system
%
\begin{equation}
\begin{split}
&\left[ \sum_{n+1<c\leq \infty}{\binom{c-1}{c-n}} \lambda_{{c},{c-n}} +  n\sigma + n\theta\nu_1 + n\theta\nu_0 \right] \alpha_n\\
&\qquad = \sum_{n+1<c\leq \infty}{\binom{c-1}{c-n}} \lambda_{{c},{c-n}} \alpha_{c-1} + n\sigma \alpha_{n-1} + n\theta\nu_1 \alpha_{n+1}, \quad n \geq 1, 
\end{split}
\label{equationproofrecursion1}
\end{equation}
again using the convention \eqref{convention}. Being tail probabilities, the  $\alpha_n$, $n \geq 0$, are monotone, with $\alpha_0=1$, and $\alpha_{\infty} := \lim_{j \to \infty} \alpha_j=0$. Together with these boundary conditions, 
Eq.~\eqref{equationproofrecursion1}  divided by $n$ gives the system \eqref{thmrecursionanequation} with $a_n$ replaced by  $\alpha_n$.

To prove uniqueness, let $(\alpha_n)$ be as above, $(a_n)$ be a solution of \eqref{thmrecursionanequation}, and put $b_n:=a_{n-1}-\alpha_{n-1}$. Then we have the boundary conditions $b_1=0$ and $b_n \to 0$ for $n \to \infty$. In addition, since both  $(\alpha_{n-1})_{2\le n < \infty}$ and $(a_{n-1})_{2\le n < \infty}$ are $G_D$-harmonic,  $(b_n)_{2 \le n < \infty}$ is $G_D$-harmonic as well.
Let $T(k):=\min \{t \geq 0 : D_t \in \{1,k ,k+1, \ldots\}\}$. Note that $T(k)$ is finite a.s. for every $k>1$.
Since, given $D_0= \ell$,  $(b_{D_t})_{t\ge 0}$ is a bounded martingale, due to the optional stopping theorem we have $b_\ell = \mathbf{E}[b_{D_{T(k)}}\mid D_0=\ell]$ for all $k > 1$. Because  $b_{D_{T(k)}} \to 0$ as $k \to \infty$, by dominated convergence this implies $b_\ell=0$ for all $\ell$, and hence the desired uniqueness.
\end{proof}

\section*{Acknowledgements}
	We thank Martin M\"{o}hle for a valuable hint that helped to include the star-shaped coalescent into Theorem \ref{thmrecursionanequation}. We are also grateful to Fernando Cordero and Sebastian Hummel for fruitful discussions. Our thanks also go to two referees for a careful reading, and for comments and suggestions which helped to improve the presentation. This project received financial support from
	Deutsche Forschungsgemeinschaft (Priority Programme SPP 1590 \emph{Probabilistic
		Structures in Evolution}, grants no. BA 2469/5-1 and WA 967/4-1).

\end{document}